\documentclass[a4paper,12pt,leqno]{amsart}
\usepackage[a4paper, top=3cm, bottom=3cm, left=3cm, right=3cm]{geometry}
\usepackage{amsmath,amsfonts,amsthm,amssymb,dsfont}
\usepackage[alphabetic]{amsrefs}
\usepackage[OT4]{fontenc}
\usepackage{enumerate}
\usepackage{mathrsfs}
\usepackage{enumerate, xspace}
\usepackage[pdftex]{graphicx}
\usepackage{color}
\usepackage{tikz-cd}
\usepackage{hyperref}
\usepackage{float}

\long\def\symbolfootnote[#1]#2{\begingroup
	\def\thefootnote{\fnsymbol{footnote}}\footnote[#1]{#2}\endgroup}

\newtheorem{theorem}{Theorem}[section]
\newtheorem{lemma}[theorem]{Lemma}
\newtheorem{thm}[theorem]{Theorem}

\newtheorem{cor}[theorem]{Corollary}

\theoremstyle{definition}
\newtheorem{rem}[theorem]{Remark}
\newtheorem{defin}[theorem]{Definition}

\newcommand{\id}{\mathrm{id}}

\newcommand{\W}{\mathcal{W}}
\newcommand{\U}{\mathcal{U}}

\begin{document}
	
	\title{A Pair of Garside shadows}
	
	\author[P.~Przytycki]{Piotr Przytycki$^{\dag}$}
	
	\address{
		Department of Mathematics and Statistics,
		McGill University,
		Burnside Hall,
		805 Sherbrooke Street West,
		Montreal, QC,
		H3A 0B9, Canada}
	
	\email{piotr.przytycki@mcgill.ca}
	
	\thanks{$\dag$ Partially supported by NSERC and (Polish) Narodowe Centrum Nauki, UMO-2018/30/M/ST1/00668}

 \thanks{$\dag$ Partially supported by the National Science Foundation under Award No. 2316995.}
	
	\author[Y.~Yau]{Yeeka Yau$^\dag$}
	\address{Department of Mathematics and Statistics,
	1 University Heights,
	University of North Carolina Asheville,
	Asheville, NC 28804,
	USA }
	\email{yyau@unca.edu}
	
	\maketitle
	
	\begin{abstract}
		\noindent We prove that the smallest elements of Shi parts and cone type parts exist and form Garside shadows. The latter resolves a conjecture of Parkinson and the second author as well as a conjecture of Hohlweg, Nadeau and Williams.
        
	\end{abstract}
	
	\section{Introduction}
	\label{sec:introd}
A \emph{Coxeter group} $W$ is a group generated by a finite set $S$ subject only to relations $s^2=1$
for $s\in S$ and $(st)^{m_{st}}=1$ for $s\neq t\in S$, where $m_{st}=m_{ts}\in \{2,3,\ldots,\infty\}$. Here the convention is that $m_{st}=\infty$
means that we do not impose a relation between $s$ and~$t$. By $X^1$ we denote the \emph{Cayley graph} of $W$, that is, the graph with vertex set $X^0=W$
and with edges (of length $1$) joining each $g\in W$ with $gs$, for $s\in S$. For $g\in W$, let $\ell(g)$ denote
the \emph{word length} of $g$, that is, the distance in $X^1$ from $g$ to $\id$.
We
consider the action of $W$ on $X^0=W$ by left multiplication. This induces an action of $W$ on $X^1$.

For $r\in W$ a conjugate of an element of $S$, the \emph{wall}
$\mathcal W_r$ of $r$ is the fixed point set of~$r$ in $X^1$. We call $r$ the \emph{reflection} in $\mathcal W_r$ (for fixed $\mathcal W_r$ such $r$
is unique). Each wall~$\mathcal W$ separates $X^1$ into two components, called \emph{half-spaces}, and a geodesic edge-path in~$X^1$ intersects~$\mathcal W$ at most once \cite[Lem~2.5]{Ronan_2009}. Consequently, the distance in~$X^1$ between $g,h\in W$ is the number of walls separating $g$ and $h$. 

We consider the partial order $\preceq$ on $W$ (called the `weak order' in algebraic combinatorics), where $p\preceq g$ if $p$ lies on a
geodesic in $X^1$ from $\id$ to $g$. Equivalently, there is no wall separating $p$ from both $\id$ and $g$.

\medskip

\noindent \textbf{Shi parts.} Let $\mathcal E$ be the set of walls $\mathcal W$ such that there is no wall separating $\mathcal W$ from  $\id$ (these walls correspond to so-called `elementary roots'). The components of $X^1\setminus \bigcup \mathcal E$ are \emph{Shi components}. For a Shi component $Y$, we call $P = Y \cap X^0$ the corresponding \emph{Shi part}.

Our first result is the following.

\begin{thm} 
\label{thm:Shi}
Let $P$ be a Shi part. Then $P$ has a smallest element with respect to~$\preceq$.
\end{thm}

Theorem~\ref{thm:Shi} was proved independently in a more general form by Dyer, Fishel, Hohlweg and Mark in \cite[Theorem 1.1(1)]{DHFM}). 
Here we give a short proof following the lines of the proof of a related result of the first author and Osajda \cite[Thm~2.1]{OP}.

In \cite{Shi_1987}, Shi proved Theorem~\ref{thm:Shi} for affine $W$. The family $\mathcal E$, which is finite by \cite{Brink1993}, has been extensively studied ever since and has become an important object in algebraic combinatorics, geometric group theory and representation theory. See for example see the survey article \cite{fishel2020}.

By \cite{Brink1993}, Shi parts are in correspondence with the states of an automaton recognising the language of reduced words of the Coxeter group. This partition of a Coxeter group is thus one of the primary examples of `regular' partitions, see \cite{Parkinson-Yau_2022}.

For $g\in W$, let $m(g)$ be the smallest element in the Shi part containing $g$, guaranteed by Theorem~\ref{thm:Shi}. 
Let $M\subset W$ be the set of elements of the form $m(g)$ for $g\in W$. 

The $\emph{join}$ of $g,g'\in W$ is the smallest element $h$ (if it exists) satisfying $g\preceq h$ and $g'\preceq h$. A subset $B\subseteq W$ is a \emph{Garside shadow} if it contains $S$, contains $g^{-1}h$ for every $h\in B$ and $g\preceq h$, and contains the join, if it exists, of every $g,g'\in B$.

\begin{thm}
\label{thm:Shi2}
$M$ is a Garside shadow.
\end{thm}

Theorem~\ref{thm:Shi2} was also obtained in \cite[Thm 1.1(2)]{DHFM}, where the authors showed that $M$ is the set of so-called `low elements' introduced in \cite{DH}. We give an alternative  proof using `bipodality', a notion introduced in \cite{DH} and rediscovered in \cite{OP}.

\medskip

\noindent \textbf{Cone type parts.} For each $g\in W$, let $T(g)=\{h\in W \ | \ \ell(gh)=\ell(g)+\ell(h)\}$. For $T\subset W$, the \emph{cone type part} $Q(T)\subset W$ is the set of all $g^{-1}$ with $T(g)=T$. In other words, $Q(T)$ consists of $g$ such that $T$ is the set of vertices on geodesic edge-paths starting at $g$ and passing through $\id$ that appear after $\id$, including $\id$. 

We obtain a new proof of the following.

\begin{thm} 
\label{thm:conetype}
{\cite[Thm~1]{Parkinson-Yau_2022}} Let $Q$ be a cone type part. Then $Q$ has a smallest element with respect to $\preceq$.
\end{thm}

For $g\in W$, let $\mu(g)$ be the smallest element in the cone type part containing $g$. Let $\Gamma\subset W$ be the set of elements of form $\mu(g)$ for $g\in W$
These elements are called the \textit{gates} of the cone type partition in \cite{Parkinson-Yau_2022}. 

We also obtain the following new result, confirming in part {\cite[Conj~1]{Parkinson-Yau_2022}}.

\begin{thm} 
\label{thm:conetype2}
For any $g,g'\in \Gamma$, if the join of $g$ and $g'$ exists, then it belongs to~$\Gamma$. 
\end{thm}

By \cite[Prop~4.27(i)]{Parkinson-Yau_2022}, this implies that $\Gamma$ is a Garside shadow. Furthermore, $\Gamma$ is the set of states of a the minimal automaton (in terms of the number of states) recognising the language of reduced words of a Coxeter group. This verifies \cite[Conj~1]{HNW}.
\\ \par
The paper is organised as follows. In Section~\ref{sec:shi_components} we discuss `bipodality' and use it to prove Theorem~\ref{thm:Shi} and Theorem~\ref{thm:Shi2}. In Section~\ref{sec:cone_type_components} we focus on the cone type parts and give the proofs of Theorem~\ref{thm:conetype} and Theorem~\ref{thm:conetype2}.

\medskip

\noindent \textbf{Acknowledgements.} We thank Christophe Hohlweg and Damian Osajda for discussions and feedback.

\section{Shi parts}
\label{sec:shi_components}

The following property was called \emph{bipodality} in \cite{DH}. It was rediscovered in \cite{OP}.
	
\begin{defin} Let $r,q\in W$ be reflections. Distinct walls $\W_r,\W_{q}$ \emph{intersect}, if $\W_r$ is not contained in a half-space
for $\W_{q}$ (this relation is symmetric). Equivalently, $\langle r,q\rangle$ is a finite group. We say that such $r,q$ are
\emph{sharp-angled}, if $r$ and $q$ do not commute and $\{r,q\}$ is conjugate into $S$. In particular, there is a component of $X^1\setminus (\W_r\cup
\W_q)$ whose intersection $F$ with $X^0$ is a fundamental domain for the action of $\langle r,q\rangle$ on~$X^0$. We call such $F$ a \emph{geometric
fundamental domain for $\langle r,q\rangle$}.
\end{defin}

\begin{lemma}[
{\cite[Lem~3.2]{OP}, special case of \cite[Thm~4.18]{DH}}]
\label{lem:key} Suppose that reflections $r,q\in W$ are sharp-angled, and that $g\in W$ lies in a geometric fundamental domain for $\langle
r,q\rangle$. Assume that there is a wall $\U$ separating $g$ from $\W_r$ or from $\W_q$. Let $\W'$ be a wall distinct from $\W_r,\W_q$ that is the
translate of $\W_r$ or $\W_q$ under an element of $\langle r,q\rangle$. Then there is a wall $\U'$ separating $g$ from $\W'$.
\end{lemma}

\begin{figure}[H]
    \centering
    
    \begin{tikzpicture}[scale=0.7]
    \draw (1.3,5) -- (1.3,-5);
    \node at (1.4,5.2) {$\mathcal{W}_r$};
    \draw (4,4) -- (-5, -5);
    \node at (4.3,4.3) {$\mathcal{W}_q$};
    \draw (-2.2,4.8) -- (5.8,-3.2);
    \node at (-2,5.3) {$\mathcal{W}'$};
    \draw (0.1,5) -- (0.1,-5);
    \node at (0.2,5.4) {$\mathcal{U}$};
    \draw (-3.8,4) -- (4.2,-4);
    \node at (-4,3.6) {$\mathcal{U}'$};
    \node at (-0.8,-2.9) {$g$};
    \end{tikzpicture}
    
    \caption{Lemma~\ref{lem:key} for the case $m_{rq} = 4$}
    \label{fig:enter-label}
\end{figure}
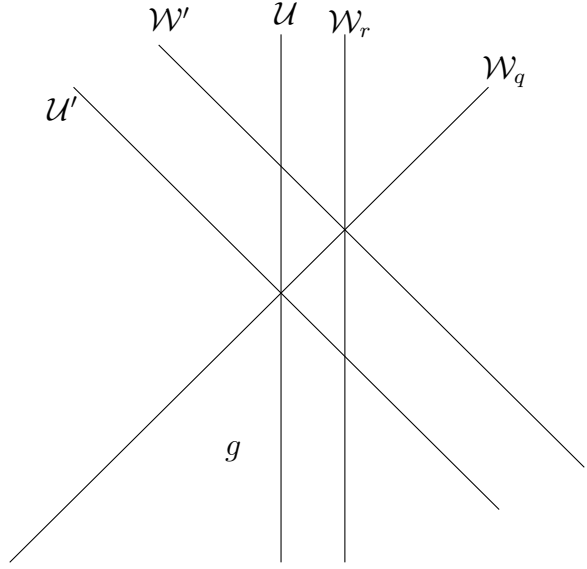

The following proof is surprisingly the same as that for a different result \cite[Thm~2.1]{OP}. 

\begin{proof}[Proof of Theorem~\ref{thm:Shi}] Let $P=Y\cap X^0$, where $Y$ is a Shi component.
It suffices to show that for each $p_0,p_n\in P$ there is $p\in P$ satisfying $p_0\succeq p\preceq p_n$. Let $(p_0,p_1,\ldots,p_n)$ be the
vertices of a geodesic edge-path $\pi$ in $X^1$ from $p_0$ to~$p_n$, which lies in $Y$. Let $L=\max_{i=0}^n\ell(p_i)$.

We will now modify $\pi$ and replace it by another embedded edge-path from $p_0$ to~$p_n$ with vertices in $P$, so that there is no
$p_i$ with $p_{i-1}\prec p_i\succ p_{i+1}$. Then we will be able to choose $p$ to be the smallest $p_i$ with respect to $\preceq$.

If $p_{i-1}\prec p_i\succ p_{i+1}$, then let $\W_r,\W_q$ be the (intersecting) walls separating $p_i$ from $p_{i-1},p_{i+1}$, respectively. Moreover, if $r$ and~$q$ do not commute, then $r,q$ are sharp-angled, with $\id$
in a geometric fundamental domain for $\langle r,q\rangle$. We claim that all the elements of the \emph{residue} $R=\langle r,q\rangle (p_i)$ lie in~$P$.

Indeed, since $p_{i-1},p_{i+1}$ are both in $P$, we have that $\W_r,\W_q\notin \mathcal E$. 
It remains to justify that each wall $\W'\neq \W_r,\W_q$ that is the translate
of $\W_r$ or $\W_q$ under an element of $\langle r,q\rangle$ does not belong to $\mathcal E$. We can thus assume that $r$ and $q$ do not commute, since otherwise there is no
such~$\W'$. Since $\W_r\notin \mathcal E$, there is a wall~$\U$ separating $\id$ from~$\W_r$. By Lemma~\ref{lem:key}, there is a wall $\U'$ separating $\id$
from~$\W'$, justifying the claim.

We now replace the subpath $(p_{i-1},p_i,p_{i+1})$ of $\pi$ by the second embedded edge-path with vertices in the residue~$R$ from $p_{i-1}$ to~$p_{i+1}$. Since all the elements of $R$ are $\prec p_i$ \cite[Thm~2.9]{Ronan_2009}, this decreases the complexity
of $\pi$ defined as the tuple $(n_L,\ldots,n_2,n_1)$, where $n_j$ is the number of $p_i$ in $\pi$ with $\ell(p_i)=j$, with lexicographic
order. After possibly removing a subpath, we can assume that the new edge-path is embedded. After finitely many such modifications, we obtain the desired path.
\end{proof}

\begin{lemma} 
\label{lem:shadow}
For $g\preceq h$, we have $m(g)\preceq m(h)$.
\end{lemma}
\begin{proof} Let $k$ be the minimal number of distinct Shi components traversed by a geodesic edge-path $\gamma$ from $h$ to $g$. We proceed by induction on $k$, where for $k=1$ we have $m(g)=m(h)$.
Suppose now $k>1$.
If a neighbour $f$ of $h$ on $\gamma$ lies in the same Shi component as $h$, then we can replace $h$ by $f$. Thus we can assume that $f$ lies in a different Shi component than $h$. Consequently, the wall $\W_r$ separating $h$ from~$f$ belongs to $\mathcal E$. Since $g\preceq f$, by the inductive assumption we have $m(g)\preceq m(f)$. Thus it suffices to prove $m(f)\preceq m(h)$.

In the first case, where for every neighbour $h'$ of $h$ on a geodesic edge-path from~$h$ to $\id$, the wall separating $h$ from $h'$ belongs to~$\mathcal E$, we have $h=m(h)$ and we are done. Otherwise, let $\W_q$ be such a wall separating $h$ from $h'$ outside $\mathcal E$. 
If $r$ and~$q$ do not commute, then $r,q$ are sharp-angled, with $\id$
in a geometric fundamental domain for~$\langle r,q\rangle$. By Lemma~\ref{lem:key}, among the walls in $\langle r,q\rangle\{\W_r,\W_q\}$ only $\W_r$ belongs to~$\mathcal E$. Let $\bar h, \bar f$ be the vertices opposite to $f,h$ in the residue $\langle r,q\rangle h$. We have $m(\bar h)=m(h), m(\bar f)=m(f)$. Replacing $h,f$ by $\bar h,\bar f$, and possibly repeating this procedure finitely many times, we arrive at the first case.
\end{proof}

Lemma~\ref{lem:shadow} has the following immediate consequence.

\begin{cor} 
\label{cor:Shi}
For any $g,g'\in M$, if the join of $g$ and $g'$ exists, then it belongs to~$M$. 
\end{cor}

For completeness, we include the proof of the following.

\begin{lemma}[{\cite[Prop 4.16]{DH}}]
\label{lem:suffix_shi}
    For any $h \in M$ and $g \preceq h$, we have $g^{-1}h \in M$.
\end{lemma}
\begin{proof}
  For any neighbour $h'$ of $h$ on a geodesic edge-path from $h$ to $g$, the wall $\W$ separating $h$ from $h'$ belongs to $\mathcal E$. Consequently, we also have $g^{-1}\W\in \mathcal E$, and so $g^{-1}h\in M$. 
\end{proof}

Also note that for each $s\in S$, we have $\mathcal W_s\in \mathcal E$ and so $m(s)=s$ implying $S\subset M$.
Thus Corollary~\ref{cor:Shi} and Lemma~\ref{lem:suffix_shi} imply Theorem~\ref{thm:Shi2}.

\section {Cone type parts}
\label{sec:cone_type_components}
Let $T=T(g)$ for some $g\in W$. We denote by $\partial T$ the set of walls separating adjacent vertices $h\in T$ and $h'\notin T$. In particular, the walls in $\partial T$ separate $\id$ from~$g^{-1}$.

We note that one of the primary differences between the cone type parts and the Shi parts is that the cone type parts do not correspond to a `hyperplane arrangement'. See for example Figure~\ref{fig:conetype_arrangement334}.

\begin{figure}[H]
    \centering
    \includegraphics[scale=0.15]{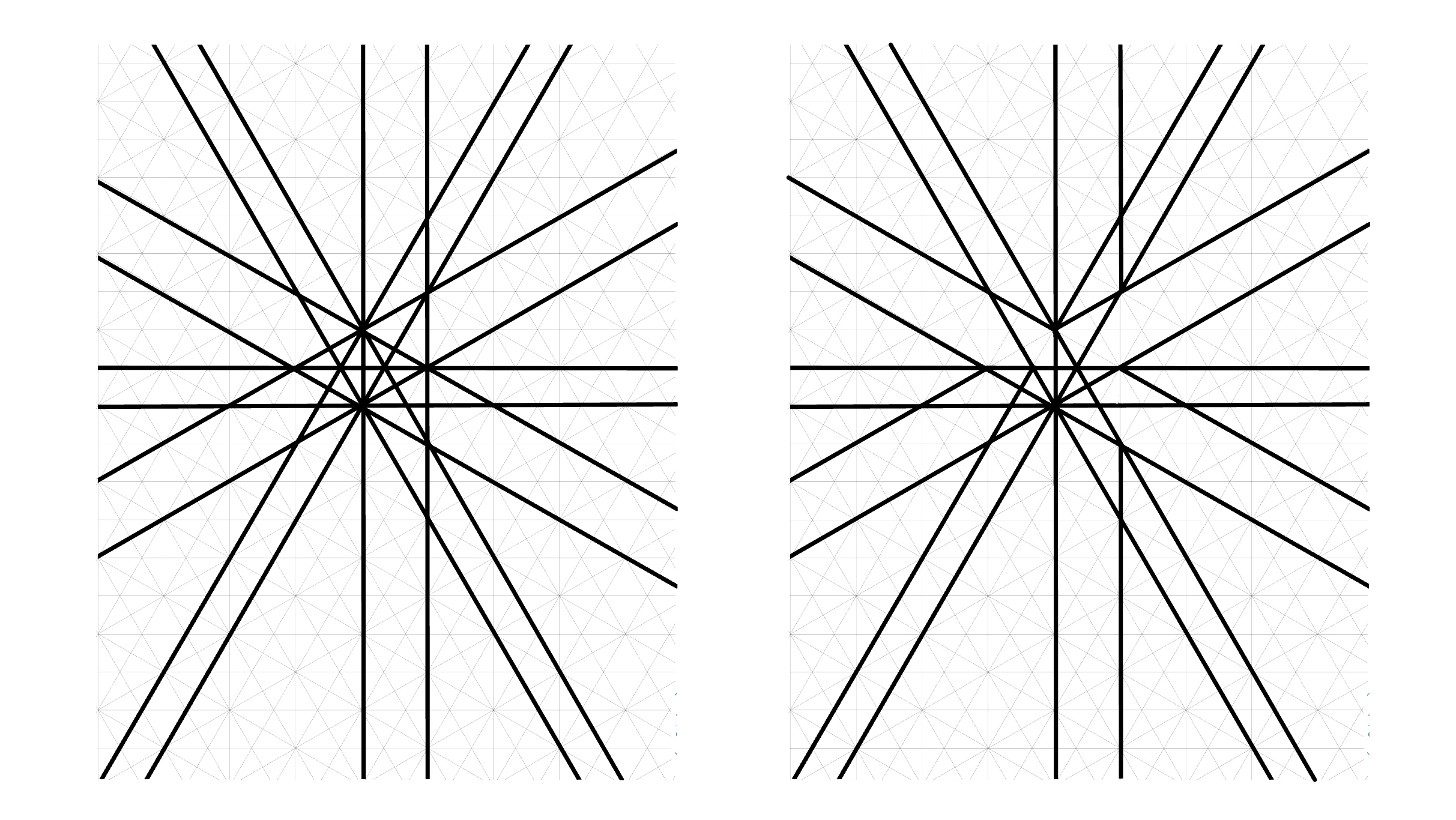}
    \caption{Shi parts and cone type parts for the Coxeter group of type $\widetilde{G}_2$}
    \label{fig:conetype_arrangement334}
\end{figure}

\begin{rem}
\label{rem:convex}
Note that for $g,g'\in Q(T)$ any geodesic edge-path from $g$ to $g'$ has all vertices~$f$ in~$Q(T)$.
Indeed, for $h\in T$, any wall separating $\id$ from $f$ separates $\id$ from $g$ or~$g'$ and so it does not separate $\id$ from $h$. Thus $h\in T(f^{-1})$ and so $T\subseteq T(f^{-1})$. Conversely, if we had $T\subsetneq T(f^{-1})$ then there would be a vertex $h\in T$ with a neighbour $h'\in T(f^{-1})\setminus T$ separated from $h$ by a wall $\W$ (in $\partial T$) that does not separate $h$ from $f$. The wall $\W$ would not separate $h'$ from $g$ or $g'$, contradicting $h'\notin T(g^{-1})$ or $h'\notin T(g'^{-1})$. See also \cite[Thm~2.14]{Parkinson-Yau_2022} for a more general statement. 
\end{rem}

\begin{proof}[Proof of Theorem~\ref{thm:conetype}]
The proof is identical to that of Theorem~\ref{thm:Shi}, with $P$ replaced by $Q$.
The vertices of a geodesic edge-path $\pi$ in $X^1$ from $p_0$ to $p_n$ belong to $Q$ by Remark~\ref{rem:convex}. We also make the following change in the proof of the claim that all the elements of $R=\langle r,q\rangle (p_i)$ lie in~$Q$.
Namely, since $T=T(p^{-1}_i)$ equals $T(p^{-1}_{i-1})$, we have $\W_r\notin \partial T$. Analogously we obtain $\W_q\notin \partial T$. If $r$ and $q$ do not commute, we have that $T$ is contained in a geometric fundamental domain for $\langle r,q\rangle$, and so we also have $\W'\notin \partial T$ for any~$\W'$ that is a translate
of $\W_r$ or $\W_q$ under an element of $\langle r,q\rangle$. This justifies the claim.
\end{proof}

\begin{proof}[Proof of Theorem~\ref{thm:conetype2}]
The proof structure is similar to that of Lemma~\ref{lem:shadow}. We need to justify that for $g\preceq h$, we have $\mu(g)\preceq \mu(h)$, where we induct on the minimal number~$k$ of distinct cone type components traversed by a geodesic edge-path $\gamma$ from $h$ to~$g$. 
Suppose $k>1$, and let $Q=Q(T)$ be the cone type component containing~$h$. 
If a neighbour~$f$ of $h$ on $\gamma$ lies in~$Q$, then we can replace $h$ by $f$. Thus we can assume $f\notin Q$. Consequently, the wall $\W_r$ separating $h$ from~$f$ belongs to $\partial T$. Since $g\preceq f$, by the inductive assumption we have $\mu(g)\preceq \mu(f)$. Thus it suffices to prove $\mu(f)\preceq \mu(h)$.

If for every neighbour $h'$ of $h$ on a geodesic edge-path from~$h$ to $\id$, the wall separating $h$ from $h'$ belongs to $\partial T$, we have $h=\mu(h)$ and we are done. Otherwise, let $\W_q$ be such a wall separating $h$ from $h'$ outside $\partial T$. Let $\bar h, \bar f$ be the vertices opposite to $f,h$ in the residue $\langle r,q\rangle h$, and let $f'=rqh$. It suffices to prove  $\mu(\bar h)=\mu(h), \mu(\bar f)=\mu(f)$. To justify $\mu(\bar h)=\mu(h)$, or, equivalently, $\bar h\in Q$, it suffices to observe that among the walls in $\langle r,q\rangle\{\W_r,\W_q\}$ only $\W_r$ belongs to $\partial T$: Indeed, if $r$ and~$q$ do not commute, then $r,q$ are sharp-angled, with $T$
in the geometric fundamental domain $F$ for $\langle r,q\rangle$ containing $\id$. 

It remains to justify $\mu(\bar f)=\mu(f)$, or, equivalently, $T(\bar f^{-1})=\widetilde T$ for $\widetilde T=T(f^{-1})$. Since $\widetilde T\cap F=T$, to show, for example, $T(f'^{-1})=\widetilde T$, it suffices to show that the wall $\W=r\W_q$ does not belong to $\partial \widetilde T$. 

Otherwise, let $b\in \widetilde T$ be adjacent to $\W$. Then $rb\in F$ is adjacent to $\W_q$, which is outside $\partial T$. Consequently, $rb\notin T$. Thus there is a wall $\W'$ separating $\id$ from $h$ and~$rb$. Note that $\W'\neq \W_r$ and so $\W'$ separates $\id$ from $f$.
Since $\id$ lies on a geodesic edge-path from $f$ to $b$, we have that $\W'$ does not separate $\id$ from $b$. 
Thus $r\W'$ separates $r$ and $rb$ from $f,h,b$, and $\id$, since, again, $\id$ lies on a geodesic edge-path from $f$ to $b$.

Consider the distinct connected components $\Lambda_1,\Lambda_2,\Lambda_3,\Lambda_4$ of $X^1\setminus (\W_r\cup r\W')$ with $\id\in \Lambda_1,b\in \Lambda_2,r\in \Lambda_3,rb\in \Lambda_4$. 
Since $\id$ and $r$ are interchanged by the reflection $r$ and they lie in the opposite connected components, we have $r\Lambda_2\subsetneq \Lambda_1$. On the other hand, since $b$ and $rb$ lie in the opposite connected components, we have $r\Lambda_1\subsetneq \Lambda_2$, which is a contradiction. 

This proves that the wall $\W$ does not belong to $\partial \widetilde T$, and hence neither does any other wall in $\langle r,q\rangle\{\W_r,\W_q\}$. Consequently $T(\bar f^{-1})=\widetilde T$, as desired. 
\end{proof}

\begin{figure}[H]
    \begin{tikzpicture}[scale=0.8]

    \draw (-1.7,-3) -- (2,5);
    \node at (2.3,5.2) {$\mathcal{W}_r$};
    \draw (-5,2.34) -- (3.4,3.04);
    \node at (3.8,3.04) {$\mathcal{W}_q$};
    \draw (-0.6,5.2) -- (4.4,-2);
    \node at (5.5,-1.8) {$\mathcal{W} = r\mathcal{W}_q$};
    \node at (1.5,3.3) {$h$};
    \node at (1.0,3.5) {$f$};
    \node at (0.4,3.1) {$f'$};
    \node at (1.1,2.3) {$\bar{h}$};
    \node at (0.56,2.4) {$\bar{f}$};
    \node at (-1.5,0) {$id$};
    \node at (0.5,-0.8) {$r$};
    \node at (3.2,-0.7) {$b$};
    \node at (-4.2,1.8) {$rb$};

    \node at (-1.5,2.2) {$\Lambda_1$};
    \node at (-4.2, -2) {$\Lambda_4$};
    \node at (2.3, -2) {$\Lambda_3$};
    \node at (1.8,1) {$\Lambda_2$};
    
    \draw (5, -0.8) .. controls (2.5, -3) and (1.5, 2.8) .. (-1, -0.45) .. controls (-0.87, -0.33) and (-1.8,-1.5) .. (-2.3, 0) .. controls (-2.5, 1) and (-3, 2.5) .. (-3, 3.5);

    \node at (-3, 4) {$r\mathcal{W}'$};

\end{tikzpicture}
\caption{Proof of Theorem~\ref{thm:conetype2}, the case of $m_{rq} = 3$}
\label{}
\end{figure}
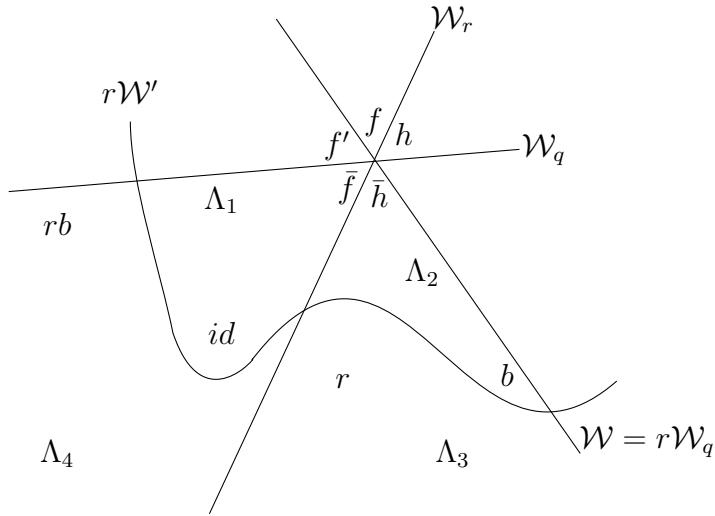

\end{document}